\documentclass[a4paper,12pt]{article}
\usepackage[top=2.5cm,bottom=2.5cm,left=2.5cm,right=2.5cm]{geometry}
\usepackage{cite,amsmath,amssymb}
    \usepackage[margin=1cm,%
                font=small,%
                format=hang,%
                labelsep=period,%
                labelfont=bf]{caption}
\usepackage[algoruled,linesnumbered]{algorithm2e}
\usepackage{algorithmic}

\usepackage{amsfonts,epsf,here,graphicx}
\usepackage{latexsym,multirow,rotating}
\usepackage{pgf,tikz,subfigure}
\usepackage{epstopdf}

\newtheorem{theorem}{\bf Theorem}[section]
\newtheorem{corollary}[theorem]{\bf Corollary}
\newtheorem{lemma}[theorem]{\bf Lemma}
\newtheorem{proposition}[theorem]{\bf Proposition}

\newtheorem{remark}[theorem]{\bf Remark}
\newtheorem{definition}[theorem]{\bf Definition}

\newtheorem{example}[theorem]{\bf Example}

\newcommand{\qed}{\hfill $\square$ \bigskip}

\begin{document}

\baselineskip=0.30in
\vspace*{40mm}

\begin{center}
{\LARGE \bf Resonantly Equivalent Catacondensed Even Ring Systems}
\bigskip \bigskip

{\large \bf Simon Brezovnik$^a$, \qquad
Niko Tratnik$^{a}$,\\ \qquad Petra \v Zigert Pleter\v sek$^{a,b}$
}
\bigskip\bigskip

\baselineskip=0.20in

\smallskip

$^a$ {\it University of Maribor, Faculty of Natural Sciences and Mathematics, Slovenia} \\
\smallskip
{\tt simon.brezovnik2@um.si, niko.tratnik@um.si, petra.zigert@um.si}
\medskip


$^b$ {\it University of Maribor, Faculty of Chemistry and Chemical Engineering, Slovenia}\\
\medskip

\bigskip\medskip

(Received \today)

\end{center}

\noindent
\begin{center} {\bf Abstract} \end{center}
In this paper we generalize the binary coding procedure of perfect matchings from catacondensed benzenoid graphs to catacondensed even ring systems (also called cers). Next, we study cers with isomorphic resonance graphs. For this purpose, we define resonantly equivalent cers. Finally, we investigate cers whose resonance graphs are isomorphic to the resonance graphs of catacondensed benzenoid graphs. As a consequence we show that for each phenylene there exists a catacondensed benzenoid graph such that their resonance graphs are isomorphic. 
\vspace{3mm}\noindent

\baselineskip=0.30in



\section{Introduction}

In the present paper we  focus  on   a  class of graphs called
even ring systems. An \textit{even ring system} is a simple bipartite 2-connected plane graph with all interior vertices of degree 3 and all boundary
vertices of degree 2 or 3 \cite{kl-br}.  If all inner faces of an even ring system  are  hexagons we obtain a well known class of molecular graphs named
  \textit{benzenoid graphs} or \textit{benzenoid systems} (see \cite{gucy-89}). They represent aromatic  hydrocarbon molecules composed of
 benzene rings. However, in the existing (both mathematical and chemical) literature, there is inconsistency in the terminology pertaining to benzenoid graphs, since in some literature it is assumed that a benzenoid graph can be embedded into the regular hexagonal lattice.

The \textit{inner dual} of a given even ring system consists of vertices corresponding
to inner faces of the system; two vertices are adjacent if and only if the corresponding
faces have a common edge. A {\em catacondensed even ring system} 
(short {\em cers}) is an even ring system so that the
inner dual is a tree $T$. An inner face of a cers is called {\em terminal} if it corresponds to a vertex of degree one in $T$.
 The  class of  cers contains chemically important structures
called  catacondensed benzenoid graps as well  another interesting molecules named phenylenes:
in particular, a {\it catacondensed benzenoid graph} $B$  is a cers where all inner faces are hexagons and
 if we add squares between all pairs of adjacent hexagons of $B$, the obtained graph is called a \textit{phenylene}. 
Threre are also some other not so well known molecules that belong to  even ring systems: for example, in \cite{cy} 
$\alpha$-$4$-catafusenes were considered and in  \cite{ra} cyclooctatetraenes. 
\smallskip

 The aromaticity of  benzenoid hydrocarbon is the consequence of the delocalization of $\pi$-electrons, i.e.\,double bounds in the molecular orbitals and it is usually  represented with the Kekul\' e structures. The interaction between 
  Kekul\' e structures is reflected in the resonance graph of a given molecular graph.  
Resonance graphs were  independently introduced by   chemists 
 (El-Basil \cite{el-basil-93/1,el-basil-93/2},  Gr{\"u}ndler \cite{grundler-82}),  as well as by 
 mathematicians, since a  Kekul\' e structure is just the perfect matching of a given  graph (Zhang, Guo, and Chen \cite{zhgu-88} used a term
 $Z$-transformation graph).

It was shown in \cite{zhgu-88}  that  the resonance graphs of benzenoid graphs are always connected, bipartite, and they are either a path or have 
girth 4. Later, some of  these properties were  obtained  for all plane (elementary) bipartite graphs (for example, see \cite{zhang-zhang}).  
Nowadays the structure of resonance graphs  is well investigated for different families of graphs.
Recently, some properties of resonance graphs or closely related concepts  were established for benzenoid graphs \cite{zigert1,zigert2}, fullerenes \cite{doslic,ra-vu,tr-zi-3,zhao}, nanotubes \cite{tr-zi,tr-zi-2}, and plane bipartite graphs \cite{che0}.

For our purposes the most important result on resonance graphs was obtained in 
 \cite{kl-br},  where the authors  proved  that the resonance 
graph of   a cers  belongs to the class of median graphs. 
This result  led  to an algorithm that assigns a unique 
and quite short binary code to every perfect matching  of a catacondensed benzenoid graph \cite{kl}. Later, the binary coding of perfect matchings was generalized to any benzenoid graph \cite{zh-la-shiu}.
 For the computer--aided manipulation with  the Kekul\'e structures 
a short representation of such a structure is very welcome.
The possibility of storing a complete information on perfect matching
 by means of a short binary string is another 
asset that may become indispensable when dealing with graphs possessing many such matchings.
 
In the next section we present some basic notation and definitions. In section 3,  the   binary coding  procedure of  perfect matchings is generalized to all cers. Moreover, cers with isomorphic resonance graphs are studied in section 4. Finally, in section 5 the cers whose resonance graphs are isomorphic to the resonance graphs of a catacondensed benzenoid graph are investigated.

\section{Preliminaries}

\noindent
The {\em distance} $d_G(u,v)$ 
between vertices $u$ and $v$ of a graph $G$ 
is defined as the usual shortest path distance. The distance between two edges $e$ and $f$ of $G$, denoted by $d_G(e,f)$ or shortly by $d(e,f)$, is defined as the distance between corresponding vertices in the line graph of $G$. Here we follow this convention because in this way the pair $(E(G),d)$ forms a metric space. On the other hand, for edges $e = ab$ and $f = xy$ of a graph $G$ it is also legitimate to set $\widehat{d}_G(e,f) = \min \lbrace d_G(a,x), d_G(a,y), d_G(b,x), d_G(b,y) \rbrace\,$. Obviously, for $e \neq f$ it holds $d_G(e,f)=\widehat{d}_G(e,f) + 1$.
\smallskip

\noindent
The {\em hypercube} $Q_n$ of dimension $n$ is defined in the following way: 
all vertices of $Q_n$ are presented as $n$-tuples $x_1x_2\ldots x_n$ where $x_i \in \{0,1\}$ for each $i \in \lbrace 1,\ldots,n \rbrace$, 
and two vertices of $Q_n$ are adjacent if the corresponding $n$-tuples differ in precisely one position. A subgraph $H$ of a graph $G$ is an  {\em isometric subgraph}
if for all $u,v\in V(H)$ it holds $d_H(u,v)=d_G(u,v)$. If a graph is isomorphic to an isometric subgraph of $G$, we say that it can be {\em isometrically embedded} in $G$. Any isometric subgraph of a hypercube is called a {\em partial cube}.
Moreover, the notation $G[X]$ is used to denote
the subgraph of $G$ induced by the set $X$.
\smallskip

\noindent
The  {\em interval} $I(u,v)$ between vertices $u$ and $v$ consists of
all vertices on a shortest path between $u$ and $v$. A  {\em median} of a triple of vertices $u$, $v$, and $w$  is a vertex
that lies in $I(u,v)\cap I(u,w)\cap I(v,w)$.
A connected graph is a {\em median graph} if every triple of its 
vertices has a unique median. Basic results about median graphs 
can be found in \cite{imkl-00}. It is well-known that every median graph is a partial cube.
\smallskip

\noindent
A {\em 1-factor} of a graph $G$ is a
spanning subgraph of $G$ such that every vertex has degree one. The edge set of a 1-factor is called a {\em perfect matching} of $G$, which is a set of independent edges covering all vertices of $G$. In chemical literature, perfect matchings are known as Kekul\'e structures (see \cite{gucy-89} for more details). 
\smallskip

\noindent
Let $G$ be a cers. We denote the edges lying on some face $F$ of $G$ by $E(F)$. The {\em resonance graph} $R(G)$ is the graph whose vertices are the  perfect matchings of $G$, and two perfect matchings $M_1,M_2$ are adjacent whenever their symmetric difference forms the edge set of exactly one inner face $F$ of $G$, i.e.\,$M_1 \oplus M_2 = E(F)$. In such a case we say that  $M_1$ can be obtained from $M_2$ by \textit{rotating the edges} of $F$.
\smallskip

\noindent
Finally, the Cartesian product $G\Box H$ of graphs $G$ and $H$ has 
the vertex set $V(G)\times V(H)$, and vertices $(a,x)$ and $(b,y)$ are 
adjacent in $G\Box H$ whenever $ab\in E(G)$ and $x=y$, 
or $a=b$ and $xy\in E(H)$.


\section{Binary coding procedure of perfect matchings of cers}

In this section we generalize the binary coding of perfect matchings from catacondensed benzenoid graphs to cers. For this purpose, some results from \cite{kl-br} are briefly repeated in the first part of the section.

Let $G$ be a cers. If $F,F'$ are adjacent faces of $G$, then the two edges on the boundary of $F$ that
have exactly one vertex on the boundary of $F'$ are called the {\em link} of $F$ to $F'$.
In \cite{kl-br} the authors proved the following proposition. 

\begin{proposition} \cite{kl-br}
Every cers has a perfect matching and given a perfect matching $M$
for every link either both edges or none belong to $M$.
\end{proposition}

In that same paper 
the structure of the resonance graph of  $G$ was established and it was proved that the resonance graph can be isometrically embedded into the hypercube $Q_n$, where $n$ is the number of inner faces of $G$. 
Let $G$ be a cers that is obtained from a 
cers $G'$ by attaching an inner face $F$ over 
an edge $e$\, that belongs also to the inner face $F'$. Suppose that $G'$ contains $n-1$ inner faces and that 
we have already embedded $R(G')$ into $Q_{n-1}$\,. We partition the 
perfect matchings  of $G$  into the sets  ${\cal F}_e(G)$\,, ${\cal 
F}_{\overline e}^{\overline{\ell}}(G)$\,, and ${\cal F}_{\overline e}^{\ell}(G)$\,, 
accordingly to the presence or  absence of edge $e$ and  link $\ell$ of $F$ to $F'$. More precisely,
${\cal F}_e(G)$ is the set of perfect matchings of $G$ that contain edge $e$, ${\cal 
F}_{\overline e}^{\overline{\ell}}(G)$ is the set of perfect matchings of $G$ that do not contain neither $e$ nor link $\ell$, and
${\cal F}_{\overline e}^{\ell}(G)$ is the set of perfect matchings of $G$ that do not contain edge $e$ but contain link $\ell$.
Then   $R(G)[{\cal F}_e(G) \cup {\cal F}_{\overline e}^{\overline{\ell}}(G)]$ of $R(G)$ is isomorphic to $R(G')$
and also $R(G)[{\cal F}_e(G)]$ and $R(G)[{\cal F}_{\overline e}^{\ell}(G)]$ are isomorphic.
 The resonance graph $R(G)$ can be embedded  into $Q_n = Q_{n-1}\Box K_2$ in such a 
way that the perfect matchings  of ${\cal F}_e(G)\cup {\cal F}_{\overline 
e}^{\overline{\ell}}(G)$ lie in one copy of $Q_{h-1}$\,,  and the perfect matchings  of 
${\cal F}_{\overline e}^{\ell}(G)$ in the other copy, where perfect matchings 
of ${\cal F}_e(G)$ and ${\cal F}_{\overline e}^{\ell}(G)$ are pairwise 
joined by  an edge.  For more details we refer to  \cite{kl-br} or  \cite{kl}.

As a consequence of that decomposition structure the following result was also proved.
\begin{theorem}\cite{kl-br}
\label{main}
The resonance graph of a cers is a median graph.
\end{theorem}

In \cite{kl} the algorithm of binary coding procedure for benzenoid graphs was presented. In this section the algorithm is generalized to an arbitrary cers. Firstly, we need to generalize the concept of linear, kinky, and branched hexagons in a catacondensed benzenoid graph (see \cite{gucy-89} for more details).

\begin{definition}
Let $F$, $F'$, and $F''$ be three inner faces of a cers such that $F$ and $F'$ have the common edge $e$ and $F',F''$ have the common edge $f$. The triple $(F,F',F'')$ is  \textbf{regular} if the distance $d_G(e,f)$ is an even number and \textbf{irregular} otherwise. See Figure \ref{regular_iregular}.
\end{definition}

\begin{figure}[h!] 
\begin{center}
\includegraphics[scale=0.7, trim=0cm 0.5cm 0cm 0cm]{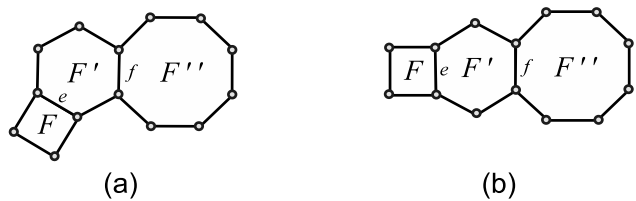}
\end{center}
\caption{\label{regular_iregular} (a) Regular and (b) irregular triple $(F,F',F'')$ of faces.}
\end{figure}

Let $G$ be a cers with $n$ inner faces. Starting from an arbitrarily chosen terminal inner face $F_1$ we can assign consecutive numbers
to each inner face to get the ordering $F_1,F_2, \ldots , F_n$. 
Let $T$ be an inner dual of cers $G$ which is a tree with $n$ vertices. The pendant vertex of $T$, which corresponds to $F_1$, is
chosen as the root of this tree. The inner faces of $G$ are
then numbered such that $F_i$ is a predecessor of $F_j$ in $T$
if and only if $i<j$. Such a numbering of inner faces is called \textit{well-ordered} and can be obtained, for example,
by the Depth-First Search algorithm (DFS) or by the
Breadth-First Search algorithm (BFS).

Suppose that we have already obtained the binary coding for the perfect matchings of a graph composed of faces $F_1,\ldots,F_{k-1}$ and that we have obtained the set of binary strings $S_{k-1}$, $k \geq 3$. 
Let $F_j$ be the face from the set $ \left\{ F_1,\ldots,F_{k-1} \right\}$ that is adjacent to $F_k$. Moreover, define the face $F_i$ as the face with the smallest index among all the adjacent inner faces of $F_j$. Then one can obtain $S_k$ by choosing one of the next options:
\begin{itemize}
\item [(a)] if $(F_{i}, F_{j}, F_{k})$ is regular, we obtain $S_k$ by inserting strings $x0$ (first two examples in Figure \ref{regular}(a)) for each $x =x_1\ldots x_{k-1} \in S_{k-1}$ and inserting $x1$ (the last example in  Figure \ref{regular}(a)) for each $x$ with $x_j=0$.

\item [(b)] if $(F_{i}, F_{j}, F_{k})$ is irregular, we obtain $S_k$ by inserting strings $x0$ (first two examples in Figure \ref{regular}(b)) for each $x = x_1\ldots x_{k-1} \in S_{k-1}$ and inserting $x1$ (the last example in  Figure \ref{regular}(b)) for each $x$ with $x_j=1$. 

%

\end{itemize}

\begin{figure}[h!] 
\begin{center}
\includegraphics[scale=0.45]{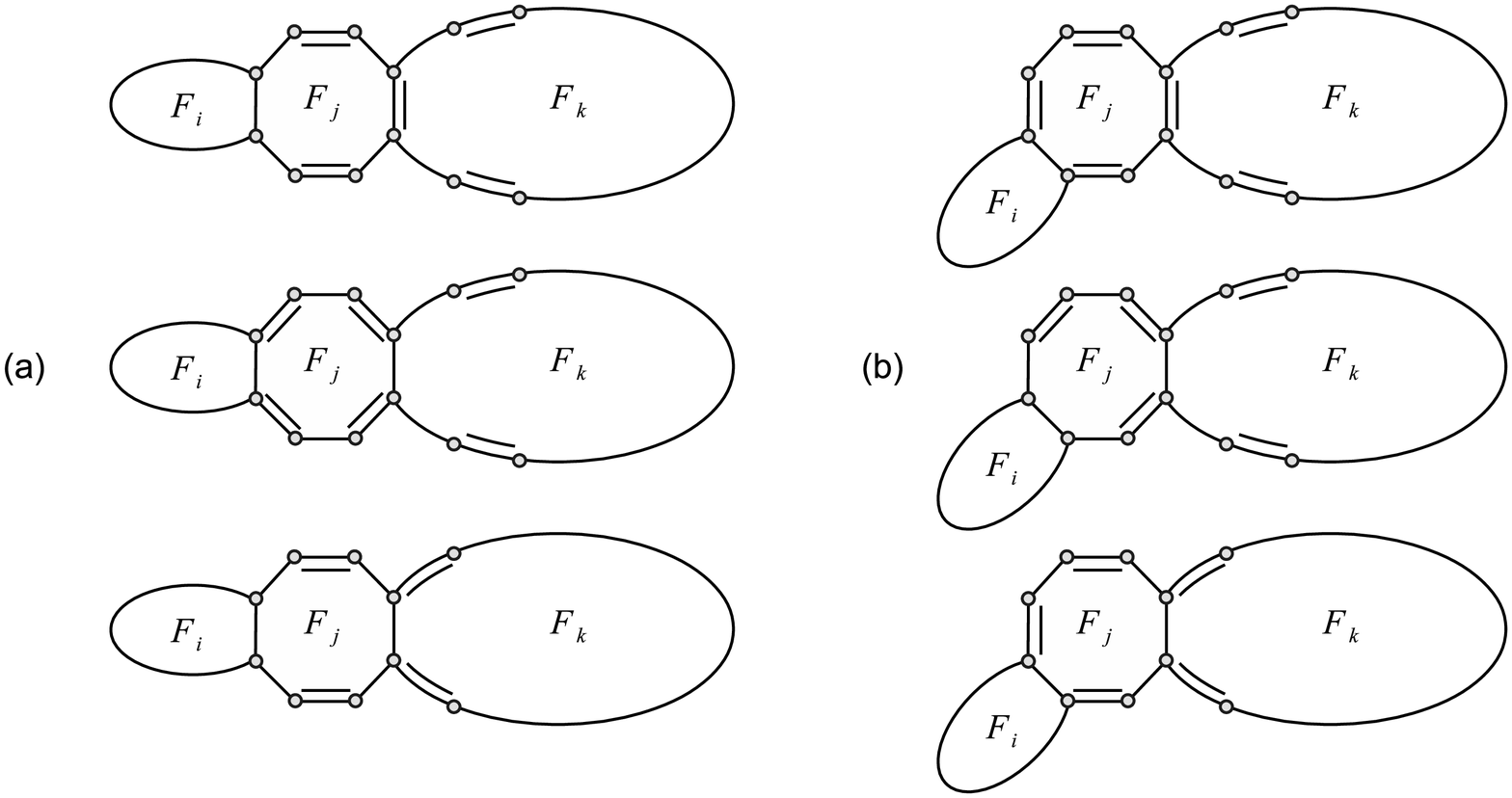}
\end{center}
\caption{\label{regular}  (a) Binary coding procedure in the regular and (b) in the irregular case.}
\end{figure}

From the above mentioned results we present the generalized version of the algorithm from \cite{kl}. For a greater transparency we mark $S_{k-1}:=S$ and $S_k:=S'$.
For a cers with two faces, $F_1$ and $F_2$, we define the binary codes $00, 01, 10$ for perfect matchings in the following way: code $00$ represents the perfect matching that contains the common edge of $F_1$ and $F_2$. Moreover, $01$ represents the perfect matching obtained from $00$ by rotating the edges in $F_2$, and $10$ is the remaining perfect matching. See Figure \ref{tri_kode} for an example.

\begin{figure}[h!] 
\begin{center}
\includegraphics[scale=0.6]{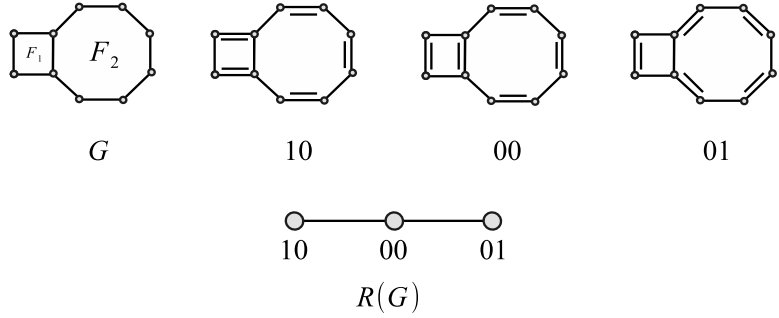}
\end{center}
\caption{\label{tri_kode} The graph $G$, the perfect matchings of $G$ and the resonance graph $R(G)$.}
\end{figure}

\begin{algorithm}[H]

    \KwIn{Graph $G$ well-ordered numbering of inner faces $F_1,F_2, \ldots , F_n$}
    \KwOut{Binary codes for all perfect matchings of a graph $G$}
   		$S:=\{00, 01, 10\}$
   	 \For{k = 3, \ldots, n}{
        	$S':=\emptyset$ \\
			set $j \in \left\{ 1, \ldots, k-1\right\} $ such that $F_jF_k \in E(T)$ \\
			$i=\min\{l;F_l F_j \in E(T) \}$ \\
  \uIf{$(F_{i}, F_{j}, F_{k})$ is regular}
  {\For {each $x \in S$}{$S':=S' \cup \{x0\}$ 
  \If{$x_j=0$}{$S':=S' \cup \{x1\}$ }
  }
  }
  \Else{\For {each $x \in S$}{$S':=S' \cup \{x0\}$ 
  \If{$x_j=1$}{$S':=S' \cup \{x1\}$}
  }
  }
    $S:=S'$
  }
  \caption{Binary coding of perfect matchings of cers}
\end{algorithm}
\bigskip

It follows that this coding procedure results in an isometric embedding of the resonance graph of a cers with $n$ inner faces into a hypercube of dimension $n$. Therefore, two perfect matchings are adjacent in the resonance graph if and only if their codes differ in precisely one position and hence, we can easily construct the resonance graph from the set of binary codes of a cers. Moreover, two cers with the same sets of codes have isomorphic resonance graphs.

For an example, we use Algorithm 1 on the cers $G$ from Figure \ref{uporaba_algoritma} and we denote its faces by $F_1,\ldots,F_5$. The resonance graphs obtained by Algorithm 1 are depicted in Figure \ref{algoritem_resonancni}. By $G_k$ we denote the subgraph of $G$ induced by the faces $F_1, \ldots, F_k$, where $k \in \lbrace 2, 3,4,5\rbrace$. Obviously, $G_5=G$ and therefore, the last graph in Figure \ref{algoritem_resonancni} is the resonance graph of $G$.
\begin{figure}[h!] 
\begin{center}
\includegraphics[scale=0.75, trim=0cm 1cm 0cm 1cm]{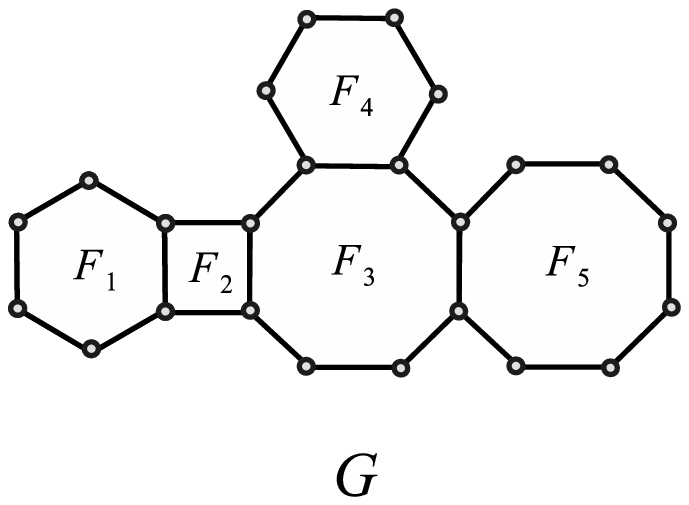}
\end{center}
\caption{\label{uporaba_algoritma} Cers $G$.}
\end{figure}

\begin{figure}[h!] 
\begin{center}
\includegraphics[scale=0.75, trim=0cm 0.5cm 0cm 0.5cm]{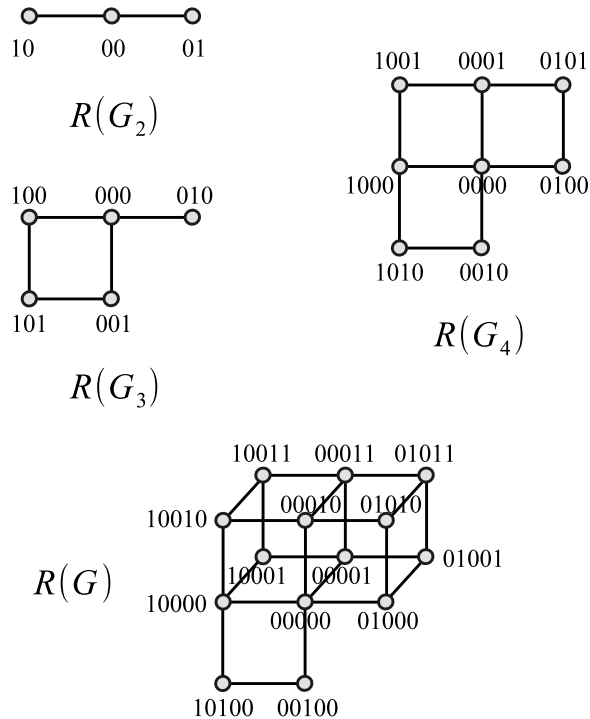}
\end{center}
\caption{\label{algoritem_resonancni} Resonance graphs $R(G_2)$, $R(G_3)$, $R(G_4)$, and $R(G)$ of graph $G$ from Figure \ref{uporaba_algoritma}.}
\end{figure}

\section{Resonantly equivalent cers}

In this section we investigate cers with isomorphic resonance graphs. To describe such graphs, we define the following transformation. If $G$ is a cers with an inner face $F$ and the outer face $F_0$, then a connected component of the graph induced by the edges in $E(F) \cap E(F_0)$ is called a {\em boundary segment}. 
\smallskip

\noindent
\textbf{Transformation 1.} Let $G$ be a cers and $P$ a boundary segment of $G$. A cers $G'$ is obtained from $G$ by subdividing edges of $P$ an even number of times or reversely, smoothing an even number of vertices of $P$ (see Figure \ref{Transformation_1}).

\begin{figure}[h!] 
\begin{center}
\includegraphics[scale=0.6]{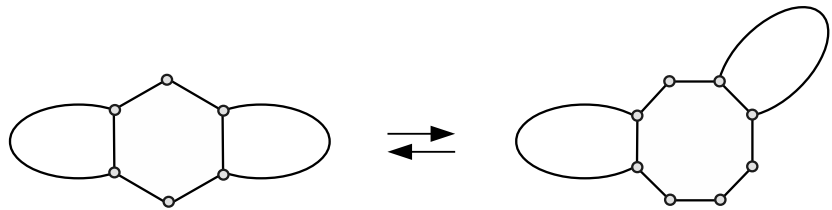}
\end{center}
\caption{\label{Transformation_1} Subdividing (from left to right) and smoothing (from right to left) an even number of times.}
\end{figure}

Using the above transformation we define a new term which is essential for the rest of the paper.

\begin{definition}
Let $G$ and $H$ be two cers. Then $G$ is \textbf{resonantly equivalent} to $H$ if $G$ can be obtained from $H$ by successively applying Transformation 1. In such a case we write $G \overset{R}{\sim} H$.
\end{definition}

\noindent
Resonantly equivalent cers are shown in Figure \ref{structurally_equivalent}
\begin{figure}[h!] 
\begin{center}
\includegraphics[scale=0.6]{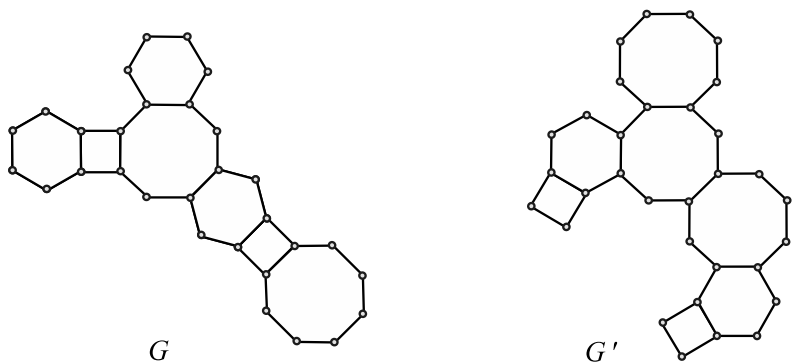}
\end{center}
\caption{\label{structurally_equivalent} Resonantly equivalent cers $G$ and $G'$.}
\end{figure}

\begin{remark}
Relation $\overset{R}{\sim}$ is an equivalence relation on the set of all cers. If $G$ and $H$ are resonantly equivalent, then both graphs have the same number of faces which differ only in lengths of its boundary segments.
\end{remark}

The following lemma is needed for the main result of this section.
\begin{lemma} \label{triple_reg}
Let $G$ and $H$ be resonantly equivalent cers and let $F_1$, $F_2$, $F_3$ be three inner faces of $G$ such that $F_1,F_2$ have a common edge and $F_2,F_3$ have a common edge. If we denote by $F_1', F_2', F_3'$ the corresponding faces of $H$, respectively, then the triple $(F_1,F_2,F_3)$ is regular if and only if the triple $(F_1',F_2',F_3')$ is regular.
\end{lemma}

\begin{proof}
Let $e \in E(F_1) \cap E(F_2)$, $f \in E(F_2) \cap E(F_3)$, $e' \in E(F_1') \cap E(F_2')$ and let $f' \in E(F_2') \cap E(F_3')$. Obviously, by the definition of Transformation 1 it holds that $d_G(e,f)$ is an even number if and only if $d_H(e',f')$ is an even number. Therefore, the triple $(F_1,F_2,F_3)$ is regular if and only if $(F_1',F_2',F_3')$ is regular. \qed
\end{proof}

\begin{theorem} \label{glavni}
Let $G$ and $H$ be two cers. If $G$ and $H$ are resonantly equivalent, then the resonance graph $R(G)$ is isomorphic to the resonance graph $R(H)$.
\end{theorem}

\begin{proof}
Let $F_1, \ldots, F_n$ be a well-ordered numbering of inner faces of $G$. Moreover, let $F_1', \ldots, F_n'$ be the corresponding faces of graph $H$. We prove the theorem by using induction on the number of faces. Let $G_k,H_k$, $ k \in \lbrace1, \ldots,n \rbrace$, be the graphs induced by the faces $F_1,\ldots,F_k$ and $F_1',\ldots,F_k'$, respectively. Obviously, the resonance graphs $R(G_2)$, $R(H_2)$ are both isomorphic to the path $P_2$ and therefore isomorphic to each other.
Suppose that we have already proved that for some $k \geq 3$ the sets of codes $S_{k-1}$ and $S_{k-1}'$ are the same (and therefore, $R(G_{k-1})$ is isomorphic to $R(H_{k-1})$). Let $F_j$ be the face of $G_k$ from the set $ \left\{ F_1,\ldots,F_{k-1} \right\}$ that is adjacent to $F_k$. Moreover, define the face $F_i$ as the face with the smallest index among all the adjacent inner faces of $F_j$. We define $F_j'$ and $F_i'$ from $H_k$ in the same way. By Lemma \ref{triple_reg} it follows that $(F_i,F_j,F_k)$ is regular if and only if $(F_i',F_j',F_k')$ is regular. Therefore, by using Algorithm 1 we obtain that $S_k = S_k'$. Hence, $R(G_k)$ is isomorphic to $R(H_k)$.

\noindent
By induction we obtain $R(G) = R(G_n) \cong R(H_n) = R(H)$ and the proof is complete. \qed 
\end{proof}

\section{Catacondensed benzenoid graphs and cers with isomorphic resonance graphs}

\begin{definition}
A cers $G$ is called \textbf{normal} if the following conditions hold:
\begin{itemize}
 \item [$(i)$] any inner face is adjacent to at most three other inner faces,
 \item [$(ii)$] if an inner face $F$ is adjacent to three other distinct inner faces $F_1,F_2,F_3$, then all the triples $(F_1,F,F_2), (F_1,F,F_3)$, and $(F_2,F,F_3)$ are regular.
 \end{itemize}
\end{definition}

\begin{theorem} \label{normal_equ}
If $G$ is a cers, then $G$ is normal if and only if $G$ is resonantly equivalent to a catacondensed benzenoid graph. 
\end{theorem}

\begin{proof} Let $G$ be a normal cers. To obtain a catacondensed benzenoid graph $G'$, we successively replace all the inner faces of $G$ by hexagons in the following way: when considering a particular inner face $F$ we perform Tranformation 1 on boundary segments of $F$ such that $F$ becomes a hexagon (note that this can be done since $G$ is a normal cers). Obviously, $G$ and $G'$ are resonantly equivalent.

For the other direction, suppose that $G$ is resonantly equivalent to a catacondensed benzenoid graph $G'$. Therefore, $G$ can be obtained from $G'$ by successively applying Transformation 1. Since this transformation preserves the number of adjacent inner faces of some inner face, it follows that any inner face of $G$ is adjacent to at most three inner faces.
Moreover, let $F$ be an inner face of $G$ that is adjacent to three other inner faces $F_1,F_2,F_3$. Let $H,H_1,H_2,H_3$ be the hexagons in $G'$ that correspond to $F,F_1,F_2,F_3$, respectively. Since $H$ is branched, the triples $(H_1,H,H_2)$, $(H_1,H,H_3)$, $(H_2,H,H_3)$ are all regular. Therefore, by Lemma \ref{triple_reg} all the triples $(F_1,F,F_2)$, $(F_1,F,F_3)$, $(F_2,F,F_3)$ are regular. Hence, we have proved that $G$ is normal. \qed
\end{proof}

\begin{theorem}
If $G$ is a normal cers, then the resonance graph $R(G)$ is isomorphic to the resonance graph of some catacondensed benzenoid graph.
\end{theorem}

\begin{proof}
Let $G$ be a normal cers and let $G'$ be the catacondensed benzenoid graph obtained in the same was as in the proof of Theorem \ref{normal_equ}. Since $G$ is resonantly equivalent to $G'$, by Theorem \ref{glavni} the resonance graph $R(G)$ is isomorphic to the resonance graph $R(G')$ and the proof is complete. \qed
%

\end{proof}

Obviously, phenylenes are normal cers and therefore, we obtain the following result. 
\begin{corollary}
For any phenylene $B'$ there exists a catacondensed benzenoid graph $B$ such that the resonance graph $R(B')$ is isomorphic to the resonance graph $R(B)$. 
\end{corollary}

In Figure \ref{fenilen_in_benzenoid} two resonantly equivalent graphs are depicted. Therefore their resonance graphs are isomorphic. 

 \begin{figure}[h!] 
\begin{center}
\includegraphics[scale=0.7]{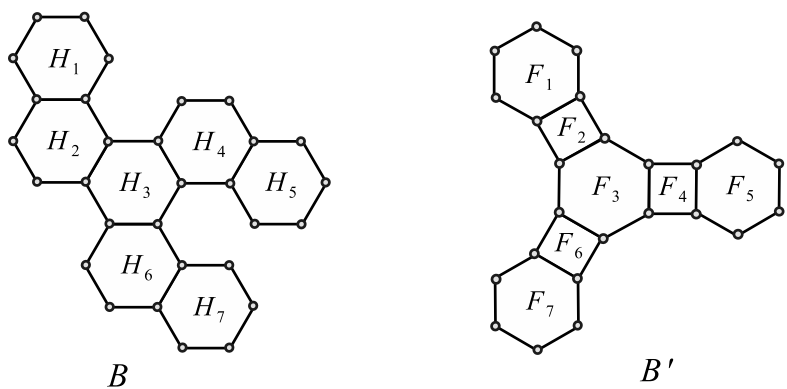}
\end{center}
\caption{\label{fenilen_in_benzenoid} Resonantly equivalent benzenoid graph $B$ and phenylene $B'$.}
\end{figure}

However, there are some cers for which the resonance graphs are not isomorphic to the resonance graph of some catacondensed benzenoid graph. In the following example we show a cers in which any inner face is adjacent to at most three other inner faces, but it has an irregular triple of faces. We show that the resonance graph of such cers is not the resonance graph of a catacondensed benzenoid graph. For this purpose, the following theorem will be used.

\begin{theorem} \label{izr_stari} \cite{zhgu-88}
Let $G$ be a benzenoid graph with a perfect matching that can be embedded into the regular hexagonal lattice. Further, let $V_1$ be the set of vertices of degree one in $R(G)$. If $R(G)$ is not a path, then it is a graph of girth 4 and $R(G) - V$ is 2-connected.
\end{theorem} 

Note that the previous theorem was proved in \cite{zhgu-88} only for benzenoid graphs that can be embedded into the regular hexagonal lattice. However, the same proof can be used for all catacondensed benzenoid graphs.

\begin{example}
Let $G$ be a cers depicted in Figure \ref{primer_neekvivalentnega}. The resonance graph $R(G)$ is constructed from the perfect matchings also shown in the same figure. Let $H$ be the subgraph obtained from $R(G)$ by removing the (unique) vertex of degree one. Obviously, $H$ contains a cut vertex and therefore, it is not 2-connected. Hence, by Theorem \ref{izr_stari}, $R(G)$ can not be the resonance graph of some catacondensed benzenoid graph.
\end{example}

 \begin{figure}[h!] 
\begin{center}
\includegraphics[scale=0.6]{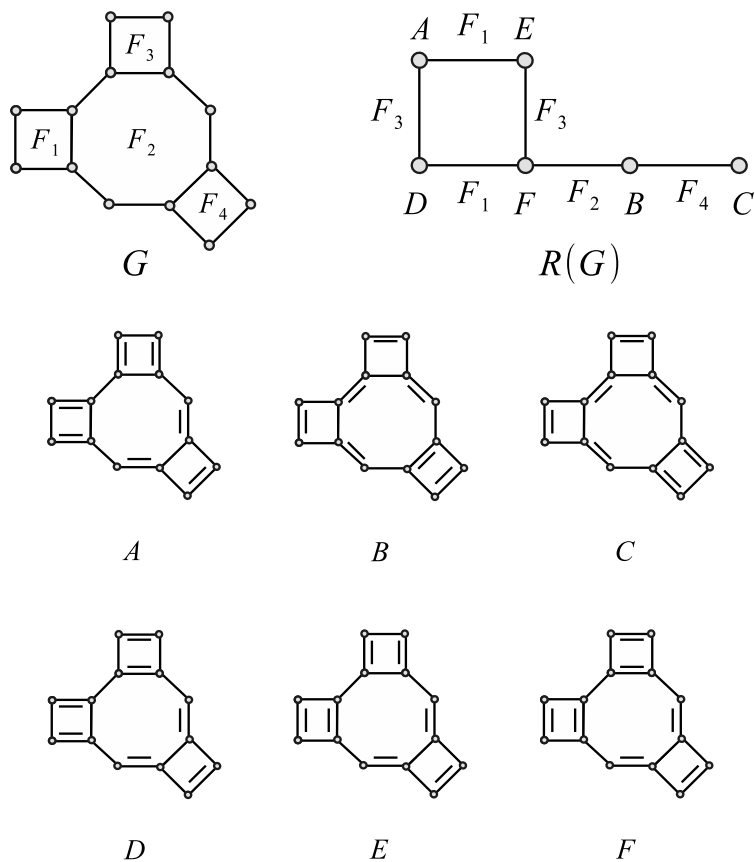}
\end{center}
\caption{\label{primer_neekvivalentnega} Graph $G$ with its perfect matchings and the resonance graph.}
\end{figure}

\section*{Acknowledgment} 

\noindent The authors Petra \v Zigert Pleter\v sek and  Niko Tratnik acknowledge the financial support from the Slovenian Research Agency (research core funding No. P1-0297 and J1-9109).

\end{document}